\documentclass[10pt]{amsart}%
\usepackage[british,UKenglish,USenglish,english,american]{babel}
\usepackage[utf8]{inputenc}
\usepackage{braket}
\usepackage{xfrac}
\usepackage{amsfonts, amssymb, amsmath, amsthm, amscd}
\usepackage{mathtools}
\usepackage{empheq}
\usepackage[left=4cm,right=4cm,top=4cm,bottom=4cm]{geometry}
\usepackage{color}
\usepackage{epsfig}
\usepackage{epic,eepic}
\usepackage[usenames,dvipsnames]{pstricks}
\usepackage{epsfig}
\usepackage{amsmath}
\usepackage{amsfonts}
\usepackage{amssymb}
\usepackage{graphicx}%
\providecommand{\U}[1]{\protect\rule{.1in}{.1in}}

\DeclarePairedDelimiterX{\norm}[1]{\lVert}{\rVert}{#1}
\newtheorem{thm}{Theorem}[section]
\newtheorem{thm2}{Theorem}
\newtheorem{prop}[thm]{Proposition}
\newtheorem{lem}[thm]{Lemma}

\newtheorem{cor}[thm]{Corollary}

\theoremstyle{definition}
\newtheorem{definition}[thm]{Definition}

\theoremstyle{remark}
\newtheorem{remark}[thm]{Remark}

\begin{document}
\title{Topological rigidity for closed hypersurfaces of elliptic space forms}
\author{Eduardo R. Longa}
\address{Instituto de Matem\'{a}tica e Estat\'{\i}stica, Universidade Federal do Rio
Grande do Sul, Porto Alegre, Brasil}
\email{eduardo.longa@ufrgs.br}
\author{Jaime B. Ripoll}
\address{Instituto de Matem\'{a}tica e Estat\'{\i}stica, Universidade Federal do Rio
Grande do Sul, Porto Alegre, Brasil}
\email{jaime.ripoll@ufrgs.br}

\begin{abstract}
We prove a topological rigidity theorem for closed hypersurfaces of the
Euclidean sphere and of an elliptic space form. It asserts that, under a
lower bound hypothesis on the absolute value of the principal curvatures, the
hypersurface is diffeomorphic to a sphere or to a quotient of a sphere by a
group action. We also prove another topological rigidity result for
hypersurfaces of the sphere that involves the spherical image of its usual
Gauss map.

\end{abstract}
\keywords{rigidity; hypersurfaces; topology; sphere; space forms}
\subjclass[2010]{Primary 53C24; 53C42}
\maketitle

\section{Introduction}

In \cite{E} J. H. Eschenburg defines an $\varepsilon$-convex hypersurface
$M^{n}$ immersed in a complete Riemannian manifold $N^{n+1}$, $n\geq2$,
$\varepsilon > 0$, as a hypersurface having all the principal curvatures with
the same sign and absolute value at least $\varepsilon$. He then proves that
if $M$ is compact, $\varepsilon$-convex and $N$ has nonnegative sectional
curvature, then $M$ is the boundary of a convex body in $N$; in particular, $M$
is diffeomorphic to an $n$-dimensional sphere.  Products of spheres
$\mathbb{S}^{j} \times \mathbb{S}^{k}$ in $\mathbb{S}^{n+1}$, $j+k=n$, show that
the hypothesis on the sign of the principal curvatures is seemingly essential. 
However, there are examples in which $M$ is an immersed sphere with nowhere 
zero principal curvatures and $M$ is not $\varepsilon$-convex (see Remark \ref{Cartan}). 

Our first result gives a sufficient condition for a closed, connected and oriented hyper\-surface $M$ of 
the round sphere $\mathbb{S}^{n+1}$ to be diffeomorphic to a sphere $\mathbb{S}^n$: the principal 
curvatures are required to be, in absolute value, greater than a function of the radius of a ball that 
contains $M$. Precisely, we have:

\begin{thm2} \label{sphere rigidity}
Let $M^n$ be a closed, connected and oriented immersed hypersurface of $\mathbb{S}^{n+1}$, 
$n \geq 2$, and let $R \in (0,\pi)$ be the radius of the smallest geodesic ball containing $M$. If the 
principal curvatures $\lambda_i$ of $M$ satisfy

\begin{align} \label{in}
\inf_{p \in M} \vert \lambda_i(p) \vert > \tan \left( \frac{R}{2} \right), \quad \forall \, i \in \{1, \dots, n \},
\end{align}

\noindent then $M$ is diffeomorphic to $\mathbb{S}^n$.
\end{thm2} 

In line with Theorem \ref{sphere rigidity}, Wang and Xia proved that $M$ is
diffeomorphic to a sphere assuming that the Gauss-Kronecker curvature of $M$
does not vanish at any point and that $M$ is contained in an open hemisphere
of $\mathbb{S}^{n+1}$ (\cite{xia06}, Theorem 1.1). It is possible to prove
Wang-Xia's result from Theorem \ref{sphere rigidity} using Beltrami's map, in
a similar way used in \cite{docarmo70}, and applying a homothetic deformation
of the hypersurface (see Remark \ref{XW} for more details). It should be noted that 
in Theorem \ref{sphere rigidity}, not only we allow the principal curvatures of the 
hypersurface to have different signs, but we also do not impose any restriction on 
the size of the geodesic ball in which the hypersurface is contained (the radius $R$ 
in the theorem can be any number in the interval $(0, \pi)$).

Our next result concerns hypersurfaces of an elliptic space form, that is, of a complete
Riemannian manifold of constant sectional curvature equal to $1$. The latter are known 
to be isometric to the quotient of $\mathbb{S}^{n+1}$ by a finite group of isometries that 
acts properly discontinuously on the sphere (see \cite{docarmo}, for example). Now, we
give a sufficient condition for the hypersurface $M$ to be covered by the sphere 
$\mathbb{S}^n$ in terms of its principal curvatures and of the distance from $M$ to the 
cut locus of a certain point.

\begin{thm2}
\label{space form rigidity} Let $\Gamma$ be a nontrivial group of isometries
of $\mathbb{S}^{n+1}$, $n \geq2$, acting properly discontinuously, and let
$\pi: \mathbb{S}^{n+1} \to\mathbb{S}^{n+1}/\Gamma$ be the canonical
projection. For $x_{0} \in\mathbb{S}^{n+1}/\Gamma$, let $p_{0} \in\pi
^{-1}(x_{0})$ and define%

\begin{align*}
r = \displaystyle \min_{g \in\Gamma\setminus\{e\} } d(p_{0}, g(p_{0})).
\end{align*}

Let $M^{n}$ be a closed and connected hypersurface of $\mathbb{S}^{n+1}%
/\Gamma$ and suppose that%

\begin{align*}
d(x, C(x_{0})) \leq R, \quad\forall\, x \in M,
\end{align*}

\noindent where $C(x_{0})$ is the cut locus of $x_{0}$ and $R \in(0, r/2)$. If
the principal curvatures $\lambda_{i}$ of $M$ satisfy%

\begin{align*}
\inf_{x \in M} \vert\lambda_{i}(x) \vert> \tan\left(  \frac{\pi- r/2 + R}{2}
\right)  = \cot\left(  \frac{r - 2R}{4} \right)  , \quad\forall \, i \in\{ 1,
\dots, n \},
\end{align*}

\noindent and if $\tilde{M}:=\pi^{-1}(M)$ has $k$ connected components, then
there is a $(|\Gamma|/k)$-to-one covering map from $\mathbb{S}^{n}$ to $M$ via
the action of $\Gamma$.
\end{thm2}

As an immediate consequence of Theorem \ref{space form rigidity}, we have the
following topological rigidity result for hypersurfaces of the projective
space $\mathbb{R}\mathbb{P}^{n+1}$:

\begin{cor}\label{rpn}
Let $M^{n}$ be a closed and connected hypersurface of
$\mathbb{R}\mathbb{P}^{n+1}$ and suppose that there exists a totally geodesic
codimension one projective space $\mathbb{R}\mathbb{P}^{n}$ of $\mathbb{R}%
\mathbb{P}^{n+1}$ such that%

\begin{align*}
d(x, \mathbb{R} \mathbb{P}^{n}) \leq R, \quad\forall\, x \in M,
\end{align*}

\noindent for some $R \in(0, \pi/2)$. If the principal curvatures $\lambda
_{i}$ of $M$ satisfy%

\begin{align*}
\inf_{x \in M} \vert\lambda_{i}(x) \vert> \tan\left(  \frac{\pi/2 + R}{2}
\right)  , \quad\forall \, i \in\{ 1, \dots, n \},
\end{align*}

\noindent then $M$ is diffeomorphic to either $\mathbb{S}^{n}$ or $\mathbb{R}
\mathbb{P}^{n}$.
\end{cor}

Isometric rigidity results for hypersurfaces with non negative $r$-mean
curvature of the sphere $\mathbb{S}^{n+1}$ have been obtained in a series of
papers beginning with De Giorgi (\cite{degiorgi65}) and, independently, Simons
(\cite{simons68}, Theorem 5.2.1) in the minimal case, then by Nomizu and Smyth
(\cite{nomizu69}, Theorem 2) for constant mean curvature hypersurfaces and
finally, by Alencar, Rosenberg and Santos (\cite{alencar04}) for constant
non-negative $r$-mean curvature hypersurfaces. Later, a topological rigidity
result was obtained by Wang and Xia (\cite{xia06}, Theorem 1.2). In all these 
results it is required that the image of the Gauss map is contained in a hemisphere 
of the sphere. Unlike these authors, we obtain a topological rigidity theorem allowing 
the Gauss image of the hypersurface to lie in a neighbourhood of a great hypersphere:

\begin{thm2}
\label{variation} Let $M^{n}$ be a closed, connected and oriented immersed
hypersurface of $\mathbb{S}^{n+1}$, $n \geq2$, with unit normal $\eta: M
\to\mathbb{S}^{n+1}$. Suppose that there exists a point $p_{0} \in
\mathbb{S}^{n+1}$ such that the spherical image of $\eta$ lies in a strip of
width $L$ around the totally geodesic hypersphere 
$T = \{ x \in \mathbb{S}^{n+1} : \langle x, p_0 \rangle = 0\}$ determined by $p_{0}$,
and that $M$ is contained in the ball of radius $R$ centered at $p_{0}$. If
the principal curvatures $\lambda_{i}$ of $M$ satisfy

\begin{align*}
\inf_{p\in M}|\lambda_{i}(p)|>\frac{\sin L}{1+\cos R},\quad\forall\,
i\in\{1,\dots,n\},
\end{align*}

\noindent then $M$ is diffeomorphic to a sphere.
\end{thm2}

The technique of our paper is elementary. The results are proved by direct
calculations using a Gauss map constructed from the parallel transport in
$\mathbb{S}^{n+1}.$

\section*{Aknowledgments}
 
We would like to thank Robert Bryant for pointing out the example in Remark 
\ref{Cartan} of an immersion of $\mathbb{S}^3$ into $\mathbb{S}^4$ which 
has nonzero Gauss-Kronecker curvature and is not $\varepsilon$-convex for 
any $\varepsilon > 0$.

\section{Gauss map}

\label{gauss}

Let $M^{n}$ be a closed, connected and oriented hypersurface of $\mathbb{S}%
^{n+1}$ with unit normal vector field $\eta: M \to\mathbb{S}^{n+1}$, and fix a
point $p_{0} \in\mathbb{S}^{n+1}$ such that $-p_{0} \not \in M$. For
non-antipodal points $p, q$ in the sphere, let $\tau_{p}^{q} : T_{p}
\mathbb{S}^{n+1} \to T_{q} \mathbb{S}^{n+1}$ be the parallel transport along
the unique geodesic joining $p$ to $q$ (we agree that $\tau_{p}^{p}$ is the
identity of $T_{p} \mathbb{S}^{n+1}$). We define a Gauss map $\gamma: M
\to\mathbb{S}^{n}$ by%

\begin{align*}
\gamma(p) = \tau_{p}^{p_{0}}(\eta(p)), \quad p \in M,
\end{align*}

\noindent where $\mathbb{S}^{n}$ is the unit sphere of $T_{p_{0}}
\mathbb{S}^{n+1}$.

\begin{definition}
Given $p \in\mathbb{S}^{n+1}$ and $v \in T_{p} \mathbb{S}^{n+1}$, define a
vector field $\tilde{v}$ on $\mathbb{S}^{n+1} \setminus\{ -p_{0} \}$ by the rule%

\begin{align*}
\tilde{v}(q) = \left(  \tau_{p_{0}}^{q} \circ\tau_{p}^{p_{0}} \right)  (v),
\quad q \neq-p_{0}.
\end{align*}

\end{definition}

Let $\overline{\nabla}$ be the Riemannian connection of $\mathbb{S}^{n+1}$.
Recall that the shape operator of $M$ in the direction of $\eta$ is the
section $A$ of the vector bundle $\operatorname{End}(TM)$ of endomorphisms of
$TM$ given by%

\begin{align*}
A_{p}(v) = - \overline{\nabla}_{v} \eta, \quad p \in M, \, v \in T_{p} M.
\end{align*}

Similarly, we define another section of $\operatorname{End}(TM)$.

\begin{definition}
\label{invariant shape operator} The invariant shape operator of $M$ is the
section $\alpha$ of the bundle $\operatorname{End}(TM)$ given by%

\begin{align*}
\alpha_{p}(v) = \overline{\nabla}_{v} \widetilde{\eta(p)}, \quad p \in M, \, v
\in T_{p} M.
\end{align*}

\end{definition}

The proposition below establishes a relationship between $\gamma$ and the
extrinsic geometry of $M$.

\begin{prop}
\label{relationship} For any $p \in M$, the following identity holds:%

\begin{align*}
\tau_{p_{0}}^{p} \circ d \gamma(p) = - \left(  A_{p} + \alpha_{p} \right)  .
\end{align*}

\end{prop}

\begin{proof}
Fix $p \in M$ and an orthonormal basis $\{v_{1}, \dots, v_{n+1} \}$ of $T_{p}
\mathbb{S}^{n+1}$ such that $v_{n+1} = \eta(p)$. The vector fields $\tilde
{v}_{1}, \dots, \tilde{v}_{n+1}$ form a global orthonormal referential of
$\mathbb{S}^{n+1} \setminus\{-p_{0}\}$, so that we can write%

\begin{align}
\eta= \sum_{i=1}^{n+1} a_{i} \tilde{v}_{i} \label{normal}%
\end{align}

\noindent for certain functions $a_{i} \in C^{\infty}(M)$. Notice that
$a_{i}(p) = 0$ for $i \in\{1, \dots, n \}$ and $a_{n+1}(p) = 1$.

For $y \in M$ we have%

\begin{align*}
\gamma(y) = \tau_{y}^{p_{0}}(\eta(y)) = \tau_{y}^{p_{0}} \left(  \sum
_{i=1}^{n+1} a_{i}(y) \tilde{v}_{i} (y) \right)  = \sum_{i=1}^{n+1} a_{i}(y)
\tau_{p}^{p_{0}}(v_{i}).
\end{align*}

\noindent Therefore, if $v \in T_{p} M$,%

\begin{align}
\tau_{p_{0}}^{p}(d \gamma(p) \cdot v) = \tau_{p_{0}}^{p} \left(  \sum
_{i=1}^{n+1} v(a_{i}) \tau_{p}^{p_{0}}(v_{i}) \right)  = \sum_{i=1}^{n+1}
v(a_{i}) v_{i}. \label{composta}%
\end{align}

\noindent From (\ref{normal}) and (\ref{composta}) we obtain%

\begin{align*}
-A_{p}(v)  &  = \overline{\nabla}_{v} \eta= \sum_{i=1}^{n+1} \overline{\nabla
}_{v}(a_{i} \tilde{v}_{i}) = \sum_{i=1}^{n+1} \left[  a_{i}(p) \overline
{\nabla}_{v} \tilde{v}_{i} + v(a_{i}) \tilde{v}_{i} (p) \right] \\
&  = \overline{\nabla}_{v} \tilde{v}_{n+1} + \sum_{i=1}^{n+1} v(a_{i}) v_{i} =
\alpha_{p}(v) + \tau_{p_{0}}^{p}(d \gamma(p) \cdot v),
\end{align*}

\noindent which gives the desired result.
\end{proof}

The next proposition gives explicit formulas for $\tau_{p}^{q}$, $\gamma$ and $\alpha$, 
obtained by straightforward computations, and hence are not presented here.

\begin{prop}
Let $p$ and $q$ be non-antipodal points in $\mathbb{S}^{n+1}$, with $p \in M$.
With the above notations, the following formulae hold:

\begin{itemize}
\item[(i)] 
\begin{align*} 
\tau_{p}^{q}(v) = - \left[  \frac{\langle v, q \rangle}{1 + \langle q, p
\rangle} \right]  (q + p) + v, \quad v \in T_{p} \mathbb{S}^{n+1}.
\end{align*}

\item[(ii)]
\begin{align*} \hspace{-1.3cm}
\gamma(p) = - \left[  \frac{\langle\eta(p), p_{0} \rangle}{1 + \langle p,
p_{0} \rangle} \right]  (p + p_{0}) + \eta(p).
\end{align*}

\item[(iii)]
\begin{align*} \hspace{-2cm}
\alpha_{p}(v) = \left[  \frac{\langle\eta(p), p_{0} \rangle}{1 + \langle p,
p_{0} \rangle} \right]  v, \quad v \in T_{p} M.
\end{align*}

\end{itemize}
\end{prop}

\section{Proofs of the Theorems}

\label{sphere}

We begin with Theorem \ref{sphere rigidity}.

\begin{proof}
[Proof of Theorem \ref{sphere rigidity}]Let $\eta: M \to\mathbb{S}^{n+1}$ be
the unit normal vector field which gives rise to the orientation of $M$, and
let $p_{0}$ be the center of a geodesic ball of radius $R$ containing $M$.
Define a function $c : M \to\mathbb{R}$ by%

\begin{align*}
c(p) = \frac{\langle\eta(p), p_{0} \rangle}{1 + \langle p, p_{0} \rangle},
\quad p \in M,
\end{align*}

\noindent and a vector field $E$ on $\mathbb{S}^{n+1}$ by%

\begin{align*}
E(p) = p_{0} - \langle p, p_{0} \rangle p, \quad p \in\mathbb{S}^{n+1}.
\end{align*}

\noindent Notice that $\langle\eta(p), E(p) \rangle= \langle\eta(p), p_{0}
\rangle$ for $p$ in $M$. Then, using the Cauchy-Schwarz inequality, we have
the following estimate for $c$:%

\begin{align*}
\left\vert c(p) \right\vert \leq\frac{\norm{\eta(p)} \norm{E(p)}}{1 + \langle
p, p_{0} \rangle} = \frac{\sqrt{ 1 - \langle p, p_{0} \rangle^{2} }}{1 +
\langle p, p_{0} \rangle} = \sqrt{ \frac{ 1 - \langle p, p_{0} \rangle}{1 +
\langle p, p_{0} \rangle} }, \quad\forall\, p \in M.
\end{align*}

\noindent Thus,%

\begin{align*}
\left\vert c(p) \right\vert \leq\sqrt{ \frac{ 1 - \cos d(p, p_{0}) }{1 + \cos
d(p, p_{0})} } = \tan\left(  \frac{d(p, p_{0})}{2} \right)  \leq\tan\left(
\frac{R}{2} \right)  , \quad\forall\, p \in M.
\end{align*}

Fix $p\in M$. Choosing an orthonormal basis of $T_{p}M$ that diagonalizes the
shape operator $A_{p}$, the matrix of $-\tau_{p_{0}}^{p}\circ d\gamma(p)$ with
respect to this basis is diagonal with entries $\lambda_{i}(p)+c(p)\neq0$ (see
Proposition \ref{relationship}). Therefore, this map is an isomorphism for
each $p\in M$, and so is $d\gamma(p)$. Since $M$ is compact, $\gamma$ is a
covering map, and since $M$ is connected with $n\geq2$, $\gamma$ is a diffeomorphism.
\end{proof}

\begin{remark}
Condition (\ref{in}) does not seem to be sharp. But it is easy to see that if
we require that

\begin{align}
\inf_{p\in M}\vert\lambda_{i}(p)\vert>\varepsilon\tan\left(  \frac{R}{2}\right)
,\quad\forall\,i\in\{1,\dots,n\} \label{ce}
\end{align}

\noindent for $\varepsilon\in(0,\sqrt{2}-1)$, then the result of the theorem may be 
false. Indeed, taking

\begin{align*}
M_{r} = \mathbb{S}^{1}(r) \times\mathbb{S}^{n-1}(s) = \left\lbrace (x,y)
\in\mathbb{R}^{2} \times\mathbb{R}^{n} : \norm{x} = r, \norm{y} = s
\right\rbrace \subset\mathbb{S}^{n+1},
\end{align*}

\noindent with $s = \sqrt{1-r^{2}}$, one may prove that the radius $R$ of the
largest open geodesic ball of $\mathbb{S}^{n+1}$ that does not intersect
$M_{r}$\ is given by%

\begin{align*}
\cos R = \min\{ r, s \}.
\end{align*}

\noindent Moreover, the principal curvatures of $M_{r}$ are $\lambda
_{1}=-\sqrt{1-r^{2}}/{r}$ and $\lambda_{2}=\cdots=\lambda_{n}=r/\sqrt{1-r^{2}%
}$. A calculation shows that one can chose $r$ so that the principal
curvatures of $M_{r}$ satisfy (\ref{ce}).
\end{remark}

\begin{remark} \label{Cartan}
We outline here a construction due to E. Cartan (\cite{cartan}) that shows the existence
of immersed $3$-spheres into $\mathbb{S}^4$ with nonzero principal curvatures and 
which are not $\varepsilon$-convex. Let $V$ be the space of traceless symmetric matrices 
of order $3$ over $\mathbb{R}$, a vector space of real dimension $5$. The group 
$\mathrm{SO}(3)$ acts on $V$ via conjugation: if $m \in V$ and $A \in  \mathrm{SO}(3)$, 
let $A \cdot m = AmA^{-1}$.  This is an irreducible representation of $\mathrm{SO}(3)$, 
and the described action leaves invariant the (positive definite) quadratic form

\begin{align*}
Q(m) = \frac{1}{6} \mathrm{tr}(m^2),
\end{align*}

\noindent as well as the cubic form

\begin{align*}
C(m) = \frac{1}{2}\det(m).
\end{align*}

Let $\mathbb{S}^4 \subset V$ be the unit $4$-sphere, defined by $\mathrm{tr}(m^2)=6$. 
Since every $m \in V$ can be diagonalized by an element of $\mathrm{SO}(3)$, one easily 
verifies that $-1\leq C(m)\leq 1$ for all $m \in \mathbb{S}^4$. The examples we announced 
are the level sets $C(m) = r$ for $\vert r \vert < 1$.  They are clearly $\mathrm{SO}(3)$-orbits, 
since the only invariants of a symmetric matrix under the $\mathrm{SO}(3)$-action are its 
eigenvalues, which are completely determined by the values of $Q(m)$ and $C(m)$ (since 
$\mathrm{tr}(m)=0$).  

The level set $C(m)=0$ is a minimal hypersurface, with one of its principal curvatures 
(necessarily constant) equal to $0$ and the other two of opposite sign.  Meanwhile, as Cartan 
shows, the level sets $C(m)= \cos (3\theta)$, for $0<\theta<\pi/6$, have three nonzero 
principal curvatures (necessarily constant) given by

\begin{align*}
\cot \left(\theta-\tfrac\pi3\right),\quad \cot\left(\theta\right),\quad \cot\left(\theta+\tfrac\pi3\right)
\end{align*}

\noindent (the first one is negative and the other two are positive). Since each such orbit is 
diffeomorphic to $\mathrm{SO}(3)/D$, where $D \cong \mathbb{Z}_2 \oplus \mathbb{Z}_2$ 
is the finite group of order $4$ consisting of the diagonal matrices, and since $\mathrm{SO}(3)$ 
is, itself, double-covered by the $3$-sphere, it follows that the simply-connected cover of each
such orbit is $8$-fold and is diffeomorphic to the $3$-sphere. Thus, we get an immersion of the 
$3$-sphere into $\mathbb{S}^4$ with the claimed properties.
\end{remark}

\begin{remark}\label{XW}
Theorem \ref{sphere rigidity} implies Theorem 1.1 of \cite{xia06},
which states that if an immersed closed and orientable hypersurface $M^{n}$
($n\geq2$) of the sphere $\mathbb{S}^{n+1}$ has non-vanishing Gauss-Kronecker
curvature and is contained in an open hemisphere, then it is diffeomorphic to
a sphere. We give here a sketch of the proof. To begin with, let $p_{0}$ be
the north pole of $\mathbb{S}^{n+1}$ and let $\mathbb{S}_{+}^{n+1}$ be the
open hemisphere centered at $p_{0}$. The Beltrami map $B:\mathbb{S}_{+}%
^{n+1}\rightarrow\mathbb{R}^{n+1}\cong T_{p_{0}}\mathbb{S}^{n+1}$ is the
diffeomorphism obtained by central projection:

\begin{align*}
B(p)=\left(  \frac{p_{1}}{p_{n+2}},\dots,\frac{p_{n+1}}{p_{n+2}}\right)
,\quad p=(p_{1},\dots p_{n+2})\in\mathbb{S}_{+}^{n+1}.
\end{align*}

\noindent For $t>0$, let $H_{t}:\mathbb{R}^{n+1}\rightarrow\mathbb{R}^{n+1}$
be the homothety $x\mapsto tx$. The map we are interested in is $C_{t}%
=B^{-1}\circ H_{t}\circ B$. After a rotation, we may suppose $M$ is contained
in $\mathbb{S}_{+}^{n+1}$. By Theorem \ref{sphere rigidity} (with
$R=\tfrac{\pi}{2}$), $M$ would be diffeomorphic to $\mathbb{S}^{n}$ if all its
principal curvatures were bigger than $1$ in absolute value. This is not
necessarily true. However, defining $M_{t}=C_{t}(M)$, it is possible to show
that if $t$ is sufficiently small, then this bound on the principal curvatures
holds for $M_{t}$ (actually, the principal curvatures of $M_t$ go to infinity
as $t$ goes to zero). So, $M_{t}$, and hence $M$, will be diffeomorphic to
$\mathbb{S}^{n}$.
\end{remark}

We now prove Theorem \ref{variation}.

\begin{proof}
[Proof of Theorem \ref{variation}]Notice that $\langle\eta(p), p_{0} \rangle=
\pm\sin d(\eta(p), T)$. So, we have the following estimate for the function
$c$ defined in the proof of Theorem \ref{sphere rigidity}:%

\begin{align*}
\vert c(p) \vert= \frac{\vert\langle\eta(p), p_{0} \rangle\vert}{1 + \langle
p, p_{0} \rangle} = \frac{\sin d(\eta(p), T)}{1 + \cos d(p, p_{0})} \leq
\frac{\sin L}{1 + \cos R}.
\end{align*}

Reasoning analogously as in the proof of that theorem, we conclude that
$\gamma: M \to\mathbb{S}^{n}$ is a global diffeomorphism.
\end{proof}

Before proving Theorem \ref{space form rigidity}, we need some facts about
fundamental domains of a group action, following \cite{ozols74}. Let $\Gamma$
be a nontrivial group of isometries of $\mathbb{S}^{n+1}$ and denote
$\Gamma\setminus\{e\}$ by $\Gamma^{\ast}$. We shall make the assumption that
$\Gamma$ acts on the sphere properly discontinuously, meaning that each point
$p\in\mathbb{S}^{n+1}$ has a neighborhood $U$ such that $U\cap g(U)=\emptyset$
for $g\in\Gamma^{\ast}$.

\begin{definition}
For $p \neq q \in\mathbb{S}^{n+1}$, define the sets%

\begin{align*}
H_{p,q} = \{ x \in\mathbb{S}^{n+1} : d(p, x) < d(q, x) \}\\
A_{p,q} = \{ x \in\mathbb{S}^{n+1} : d(p, x) = d(q, x) \}.
\end{align*}

The fundamental domain of $\Gamma$ centered at $p$ is%

\begin{align*}
\Delta_{p} = \bigcap_{g \in\Gamma^{\ast}} H_{p, g(p)}.
\end{align*}

\end{definition}

We need the following facts:

\begin{prop}
[\cite{ozols74}, Proposition 3.4]\label{prop1} For each $g \in\Gamma^{\ast}$
and $p \in\mathbb{S}^{n+1}$, $\overline{\Delta}_{p} \cap\overline{\Delta
}_{g(p)} \subset A_{p, g(p)}$.
\end{prop}

\begin{prop}
[\cite{ozols74}, Proposition 3.5]\label{prop2} For $p \in\mathbb{S}^{n+1}$,%

\begin{align*}
\partial\Delta_{p} = \bigcup_{g \in\Gamma^{\ast}} \partial\Delta_{p}
\cap\partial\Delta_{g(p)}.
\end{align*}

\end{prop}

From these, we prove a series of lemmas.

\begin{lem}
\label{small ball} For $p \in\mathbb{S}^{n+1}$, define%

\begin{align*}
r = \displaystyle \min_{g \in\Gamma^{\ast}} d(p, g(p)).
\end{align*}

\noindent Then $B_{r/2}(p) \subset\Delta_{p}$. In particular, $B_{r/2}(p)
\cap\partial\Delta_{p} = \emptyset$.
\end{lem}

\begin{proof}
Suppose that this ball is not contained in the fundamental domain centered at
$p$. Then there exists $q$ belonging to the ball and to $\partial\Delta_{p}$.
Then, from Proposition \ref{prop2}, there exists $g_{0} \in\Gamma^{\ast}$ such
that $q \in\partial\Delta_{p} \cap\partial\Delta_{g_{0}(p)}$. By Proposition
\ref{prop1}, it follows that $q \in\overline{\Delta}_{p} \cap\overline{\Delta
}_{g_{0}(p)} \subset A_{p, g_{0}(p)}$. Thus, $d(p, q) < r/2$ and $d(g_{0}(p),
q) = d(p, q) < r/2$. Hence,%

\begin{align*}
d(p, g_{0}(p)) \leq d(p, q) + d(g_{0}(p), q) < \frac{r}{2} + \frac{r}{2} = r,
\end{align*}

\noindent contrary to the definition of $r$.
\end{proof}

Let $\mathbb{S}^{n+1}/\Gamma$ be the quotient space and denote by $\pi:
\mathbb{S}^{n+1} \to\mathbb{S}^{n+1}/\Gamma$ the canonical projection. The
latter is a Riemannian covering map when we endow $\mathbb{S}^{n+1}/\Gamma$
with the suitable metric.

\begin{lem}
\label{isometry} The restriction of $\pi$ to a fundamental domain $\Delta_{p}$
is an isometry onto its image.
\end{lem}

\begin{proof}
Since $\pi$ is a local isometry, it suffices to prove that the restriction of
$\pi$ to $\Delta_{p}$ is injective. Suppose $\pi(q_{1}) = \pi(q_{2})$, with
$q_{i} \in\Delta_{p}$. Without loss of generality, suppose $d(p, q_{1}) \leq
d(p, q_{2})$. There exists $g \in\Gamma$ such that $g(q_{1}) = q_{2}$. If $g
\neq e$, then we would have%

\begin{align*}
d(p, q_{2}) < d(g(p), q_{2}) = d(g(p), g(q_{1})) = d(p, q_{1}),
\end{align*}

\noindent contrary to our assumption. Thus, $g = e$ and $q_{1} = q_{2}$.
\end{proof}

\begin{lem}
\label{antipodal} For $p \in\mathbb{S}^{n+1}$, the antipodal point of $p$ does
not belong to $\overline{\Delta}_{p}$.
\end{lem}

\begin{proof}
Suppose the contrary. Then either $-p \in\Delta_{p}$ or $-p \in\partial
\Delta_{p}$. The first case cannot occur, otherwise%

\begin{align*}
\pi= d(p, -p) < d(g(p),-p)
\end{align*}

\noindent for $g \in\Gamma^{\ast}$. So, we must have $-p \in\partial\Delta
_{p}$. By Propositions \ref{prop1} and \ref{prop2}, there exists $g_{0}
\in\Gamma^{\ast}$ such that $-p \in\partial\Delta_{p} \cap\partial
\Delta_{g_{0}(p)} \subset A_{p, g_{0}(p)}$. Thus,%

\begin{align*}
\pi= d(p, -p) = d(g_{0}(p), -p)
\end{align*}

\noindent which implies that $g_{0}(p) = p$. This is an absurd, since no
element of $\Gamma^{\ast}$ has a fixed point.
\end{proof}

From Lemma \ref{antipodal}, the next fact, from \cite{ozols74}, applies:

\begin{prop}
[\cite{ozols74}, Corollary 3.11]\label{cut} If $\overline{\Delta}_{p} \cap
C(p) = \emptyset$, then $C(\pi(p)) = \pi(\partial\Delta_{p})$, where
$C(\cdot)$ denotes the cut locus.
\end{prop}

\begin{lem}
\label{inverse image} For $p \in\mathbb{S}^{n+1}$,%

\begin{align*}
\pi^{-1} \left(  \pi(\partial\Delta_{p}) \right)  = \bigcup_{g \in\Gamma}
\partial\Delta_{g(p)}.
\end{align*}

\end{lem}

\begin{proof}
This follows from the easily verifiable fact that $g(\partial\Delta_{p}) =
\partial\Delta_{g(p)}$ (see \cite{ozols74}, Proposition 3.2 (3)).
\end{proof}

We are now ready to prove Theorem \ref{space form rigidity}.

\begin{proof}
[Proof of Theorem \ref{space form rigidity}]Since $\pi$ is a local isometry,
the principal curvatures of $M$ and $\tilde{M}$ coincide. Due to Theorem
\ref{sphere rigidity}, it thus suffices to prove that the open ball
$B_{r/2-R}(p_{0})$ does not intersect $\tilde{M}$, for then the ball
$B_{\pi-r/2+R}(-p_{0})$ contains any connected component of $\tilde{M}$. We
argue by contradiction. Suppose $q$ lies both in $B_{r/2-R}(p_{0})$ and in
$\tilde{M}$. Then $d(q, p_{0}) < r/2 - R$ and $d(\pi(q), C(x_{0})) \leq R$.
From Lemmas \ref{small ball} and \ref{isometry}, we have $d(\pi(q), x_{0}) <
r/2 - R$. Thus,%

\begin{align*}
d(x_{0}, C(x_{0}))  &  \leq d(x_{0}, \pi(q)) + d(\pi(q), C(x_{0}))\\
&  < \left(  \frac{r}{2} - R \right)  + R\\
&  = \frac{r}{2}.
\end{align*}

So, there exists $y \in C(x_{0})$ such that $d(x_{0}, y) < r/2$. By Lemma
\ref{isometry}, $d(\pi\vert_{\Delta_{p_{0}}}^{-1}(y), p_{0}) < r/2$, and by
Proposition \ref{cut} and Lemma \ref{inverse image},%

\begin{align*}
\pi\vert_{\Delta_{p_{0}}}^{-1}(y) \in\pi^{-1}(y) \subset\pi^{-1}(C(x_{0}))
\subset\pi^{-1} \left(  \pi(\partial\Delta_{p_{0}}) \right)  = \bigcup_{g
\in\Gamma} \partial\Delta_{g(p_{0})}.
\end{align*}

\noindent This contradicts Lemma \ref{small ball}, since $B_{r/2}(p_{0})
\cap\partial\Delta_{p_{0}} = \emptyset$. This concludes the proof.
\end{proof}

\end{document}